\documentclass[11pt]{article}

\usepackage{latexsym}
\usepackage{amssymb}
\usepackage{amsthm}
\usepackage{amscd}
\usepackage{amsmath}
\usepackage{verbatim}
\usepackage{xcolor}
\usepackage{enumerate}
\usepackage{graphicx}

\newtheorem{theorem}{Theorem}[section]

\newtheorem{lemma}[theorem]{Lemma}
\newtheorem*{lemma*}{Lemma}

\newtheorem*{definition*}{Definition}
\newtheorem{definition}[theorem]{Definition}
\newtheorem{proposition}[theorem]{Proposition}
\newtheorem{corollary}[theorem]{Corollary}

\theoremstyle{definition}

\theoremstyle{remark}
\newtheorem{remark}{Remark}[section]

\numberwithin{equation}{section}

\newcommand{\leqn}{\begin{equation}\label}
\def\endeqn{\end{equation}}

      \def\RR{\mathbb{R}}

\def\dim{\mathrm{dim}}

\newcommand{\res}{\hbox{{\vrule height .22cm}{\leaders\hrule\hskip .2cm}}}

\makeatletter
\renewcommand{\@makefnmark}{\mbox{\textsuperscript{}}}
\makeatother

\def\adots{\mathinner{\mkern2mu\raise0pt\hbox{.}  
\mkern2mu\raise4pt\hbox{.}\mkern1mu
\raise7pt\vbox{\kern7pt\hbox{.}}\mkern1mu}}
\def\res{\hbox{ {\vrule height .22cm}{\leaders\hrule\hskip.2cm} } }

\allowdisplaybreaks[1]

\begin{document}

\title{Singular Sets of Uniformly Asymptotically Doubling Measures}
\author{ A. Dali Nimer}
\date{}
\maketitle
\footnote{Department of Mathematics, University of Chicago, 5734, S. University Ave., Chicago, IL, 60637
E-mail address: nimer@uchicago.edu}
\footnote{Mathematics Subject Classification Primary 28A33, 49Q15}

\begin{abstract}
In the following paper, we prove a dimension bound on the singular set of a Radon measure assuming its doubling ratio converges uniformly on compact sets. More precisely, we prove that if a Radon measure is $n$-Uniformly Asymptotically Doubling, then $\dim(\mathcal{S}_{\mu}) \leq n-3$, where $\mathcal{S}_{\mu}$ is the singular set of the measure.
\end{abstract}

\section{Introduction}
In $\cite{KT}$, the authors investigate a free boundary regularity problem for harmonic measures: they study how the doubling properties of the harmonic measure of a domain determine the geometry of its boundary. Indeed, they show that under the appropriate hypotheses, if the harmonic measure of a domain is asymptotically optimally doubling, then its boundary is locally well approximated by planes.

A Radon measure $\mu$ is said to be $n$-asymptotically optimally doubling (denoted by $n$-AOD) if for small radii, it doubles like Lebesgue measure. More precisely this means that the doubling ratio given by $\frac{\mu(B_{tr}(x))}{\mu(B_r(x))}$ behave like $t^n$ for small radii $r$.

\begin{definition}
Consider a Radon measure $\mu$ on $\RR^{d}$, $\Sigma=supp(\mu)$. For a fixed integer $n$, $n \leq d$, define for $x \in \Sigma$, $r>0$ and $t \in (0,1]$ 

\begin{equation}\label{doublingratio}
R_{t}(x,r)= \frac{\mu(B_{tr}(x))}{\mu(B_r(x))}-t^n
\end{equation}
which encodes the doubling properties of $\mu$. We say $\mu$ is $n$-asymptotically optimally doubling ($n$-AOD) if for each compact set $K \subset \RR^{d}$, $x \in K$ and $t \in [\frac{1}{2},1]$, we have
\begin{equation} \label{AOD}
\lim_{r \to 0^{+}} \sup_{x \in K} |R_{t}(x,r)|=0
\end{equation}
\end{definition}

In $\cite{KT}$, the following theorem is proven.

\begin{theorem} \cite{KT} \label{mainKT}
Let $\mu$ be an $n$-asymptotically doubling measure. Then its support $\Sigma \subset \RR^{n+1}$ is Reifenberg flat with vanishing constant if $n=1,2$ and if $n \geq 3$ there exists $\delta$  such that if $\Sigma$ is $\delta$ Reifenberg flat, then $\Sigma$ is Reifenberg flat with vanishing constant.
\end{theorem}

The notion of Reifenberg flatness will not be defined precisely in this paper, but for a set to be Reifenberg flat means that it is locally well approximated by planes.

In $\cite{DKT}$ and $\cite{PTT}$, the authors show that if one assumes that the rate at which the doubling ratio of a measure approaches $t^n$ is Holder, then we obtain even more information on the regularity of its support. 

\begin{theorem} [\cite{PTT}, \cite{DKT}] \label{PTT2}
 
For each $\alpha >0$, there exists $\beta=\beta(\alpha)$ with the following property. If $\mu$ is a positive Radon measure supported on $\Sigma \subset \RR^d$ whose density ratio $\frac{\mu(B_r(x))}{r^n}$ approaches $1$ in a Holder way for small $r$, then:
\begin{itemize}
\item(1.10, $\cite{PTT}$) if $n=1,2$, $\Sigma$ is a $C^{1,\beta}$ submanifold of dimension $n$ in $\RR^{d}$.
\item(1.10, $\cite{PTT}$) if $n \geq 3$, $\Sigma$ is a $C^{1,\beta}$ submanifold of dimension $n$ in $\RR^{d}$ away from a closed set $\mathcal{S}$ such that $\mathcal{H}^{n}(\mathcal{S})=0$, where  $\mathcal{S}=\Sigma \backslash \mathcal{R}$ and $\mathcal{R}= \left\lbrace x \in \Sigma ; \limsup_{r \to 0} \theta(x,r) = 0 \right\rbrace$.
\end{itemize}
 \end{theorem}
 
One of the main insights behind those results is the fact that the doubling ratio of $\mu$ behaving like a power of $t$ implies that the tangent objects to $\mu$ are $n$-uniform. We say that a measure $\nu$ is $n$-uniform if there exists a constant $c >0$ such that for every $x$ in the support of $\nu$ and every $r>0$, $\nu(B_r(x))=cr^{n}$. To make this statement more precise, let us state the notion of pseudo-tangents to a measure which were introduced in $\cite{KT}$.

\begin{definition} 
Let $\mu$ be a doubling Radon measure in $\RR^{d}$. We say that $\nu$ is a pseudo-tangent measure of $\mu$ at the point $x \in \mbox{supp} \mu$ if $\nu$ is a nonzero Radon measure in $\RR^{d}$ and if there exists a sequence of points $x_i \in \mbox{supp} \mu$  such that $x_i \to x$ and  sequences of positive numbers $\left\lbrace r_i \right\rbrace$ and $\left\lbrace c_i \right\rbrace$ such that $r_i \downarrow 0$ and $c_i T_{{x_i},{r_i}} \sharp \mu \rightharpoonup \nu$.
\end{definition} 

\begin{theorem}[\cite{KT}]\label{pseudounif} Let $\mu$ be a Radon measure in $\mathbb{R}^{d}$ that is doubling and $n$-asymptotically optimally doubling. Then all pseudo-tangent measures of $\mu$ are $n$-uniform. 
\end{theorem}
This theorem says that a large class of measures is described asymptotically by $n$-uniform measures. Therefore, understanding the geometry of the support of $n$-uniform measures is important to describe the support of measures with ``good'' doubling ratios.

A measure being $n$-uniform in $\mathbb{R}^{d}$ implies a great deal of rigidity on its geometry. Indeed, if $n=1,2$, then it has to be flat (see $\cite{P}$). If $n=d-1$, then the measure is either flat or supported on the set $\mathcal{C} \times \mathbb{R}^{n-3}$ where $\mathcal{C}$ is the cone given by $\mathcal{C}=\left\lbrace (x_1,x_2,x_3,x_4) \; ; \; x_4^2= x_1^2+x_2^2+x_3^2\right\rbrace$ (see $\cite{KoP}$).
In $\cite{N1}$, the following result on the singular set of $n$-uniform measures is proven. 

\begin{theorem} [\cite{N1}] \label{N1}
Let $\nu$ be an $n$-uniform measure in $\RR^{d}$, $ 3 \leq n \leq d$.
Then
$$dim_{\mathcal{H}}(\mathcal{S}_{\nu}) \leq n-3.$$
\end{theorem}
Since an $n$-AOD measure has $n$-uniform pseudo-tangents, the same dimension bound should apply to its singular set. In this paper, we prove that this is indeed true. In fact, we prove that assuming the measure asymptotically doubles like any continuous function is enough to deduce such a dimension bound.

One can easily see the proof of Theorem $\ref{N1}$ works for measures that are merely uniform or uniformly distributed. 

\begin{definition} We say a measure $\nu$ is uniform if there exists a positive real-valued function $\phi$ on the positive real numbers such that for every $x$ in the support of $\nu$ and every $r>0$, $\nu(B_r(x))=\phi(r)$. 
\end{definition}

Using a theorem of Preiss (see $\cite{P}$) that states that for every uniform measure, there exists $n=\dim_{0}(\nu)$ such that $\lim_{r \to 0} \frac{\phi_{\nu}(r)}{r^{n}}$ exists, is positive and finite, we can show that $\nu$ $n$-uniform can be replaced by $\nu$ uniform with $\dim_{0} \nu = n$ in the statement of Theorem $\ref{N1}$.

With this in mind, we define the following more general notion of ``well-behavedness'' of the doubling ratio of a measure.

\begin{definition}\label{UAD}
Let $\mu$ be a Radon doubling measure in $\mathbb{R}^{d}$, $\Sigma=spt(\mu)$. We say $\mu$ is uniformly asymptotically doubling (UAD) if  there exists a continuous function $f_{\mu}: \Sigma \times \RR_{+} \to \mathbb{R}_+$, $f_{\mu}(x,1)=1$ for every $x \in \Sigma$ such that, for every $K$ compact with  $K \cap \Sigma \neq \emptyset$, and for every $\epsilon>0$, there exists $r_{K}>0$ such that:

\begin{equation}\label{UADdef}
r \leq r_{K} \implies \left| \frac{\mu(B_{tr}(x))}{\mu(B_r(x))} - f_{\mu}(x,t) \right| < \epsilon, \mbox{ for every } x \in K \cap \Sigma, \; t \in (0,1].
\end{equation}
\end{definition}

We first prove that if  a measure is UAD, then its pseudo-tangents
 are uniform.
 
\begin{theorem} \label{pseudotangentUAD}
Let $\mu$ be a uniformly asymptotically doubling measure in $\RR^{d}$. Then all pseudo-tangents of $\mu$  are uniform.
\end{theorem}

We use this to prove a bound on the singular set of UAD measures. 
The proof of Theorem $\ref{N1}$ relied on a dimension reduction argument based on the fact that for $n=3$, the geometry of the tangents to a $3$-uniform measure are well understood (see $\cite{N2}$). But a key aspect in dimension reduction arguments is that the measures have to satisfy a ``conservation of singularities under blow-ups'' property. In other words, the singular set of the measure in the argument has to blow-up to the singular set of its tangent object. We prove an analogous result for UAD measures.

\begin{theorem} \label{accumsing}
Let $\mu$ be a UAD measure.
Let $\left\lbrace x_j \right\rbrace_{j=0}^{\infty} \subset \mathcal{S}_{\mu}$, $\nu \in Tan(\mu,x_0)$, $\left\lbrace r_j \right\rbrace_{j}$ a sequence going to zero such that $\mu_{x_0,r_j} \rightharpoonup \nu$. Moreover, let $y_j= \frac{x_j-x_0}{r_j} \in B_1(0)$, $y_j \to y$.
Then $y \in \mathcal{S}_{\nu}$.
\end{theorem}

Once this theorem is proven, we can apply a dimension reduction argument to get the following final theorem.

\begin{theorem}\label{maintheorem}
Let $\mu$ be an $n$-UAD measure in $\RR^{d}$, $ 3 \leq n \leq d$.
Then
$$dim_{\mathcal{H}}(\mathcal{S}_{\nu}) \leq n-3.$$
\end{theorem}
\newpage
\section{Preliminaries}

\begin{definition}
Let $\mu$ be a measure in $\RR^{d}$. We define the support of $\mu$ to be 
\begin{equation}
\mbox{supp}(\mu)=\left\lbrace x \in \RR^{d} ; \mu(B_r(x)>0, \mbox{ for all } r>0 \right\rbrace. 
\end{equation}
Note that the support of a measure is a closed subset of $\RR^{d}$.
\end{definition}
We can define weak convergence for a sequence of Radon measures.
\begin{definition}
Let $\Phi$, $\Phi_j$, $j>0$ be Radon measures in $\RR^{d}$.
We say that $\Phi_j$ converges weakly to $\Phi$ if for every $f \in C_{c}(\RR^{d})$, the following holds:
\begin{equation}
\int f(z) d\Phi_j (z) \to \int f(z) d\Phi(z).
\end{equation} 
We denote it by $ \Phi_j \rightharpoonup \Phi$.
\end{definition}
The results in this section appear in this form in $\cite{M}$.
\begin{theorem}
Let $\Phi_j$ be a sequence of Radon measures on $\RR^{d}$.Then $\Phi_j \rightharpoonup \Phi$, if and only if for any $K$ compact subset of $\RR^{d}$ and any $G$ open subset of $\RR^{d}$ the following hold:
\begin{enumerate}
\item $\Phi(K) \geq \limsup \Phi_{j}(K).$
\item $\Phi(G) \leq \liminf \Phi_{j}(G).$
\item For any point x there exists a set $S_x \subset \RR_{+}$ at most countable such that
$\Phi_{j}(B_{\rho}(x)) \to \Phi(B_{\rho}(x))$ for every $ \rho \in \RR_{+} - S_x$ .
\end{enumerate}

\end{theorem}

\begin{theorem}\label{weakconv}
Let $\Phi_j$ be a sequence of Radon measures on $\RR^{d}$ such that
$$ \sup_{j}(\Phi_j(K)) < \infty ,$$
for all compact sets $K \subset \RR^{d}$. Then there is a weakly convergent subsequence of $\Phi_j$.
\end{theorem}

We now want to define a metric on the space of Radon measures.

\begin{definition}\label{L(r)}
Let $0<r<\infty$. We denote by $\mathcal{L}(r)$ the set of all non-negative Lipschitz functions $f$ on $\RR^{d}$ with $spt(f) \subset B_r(0)$ and with $Lip(f) \leq 1$. For Radon measures $\Phi$ and $\Psi$ on $\RR^d$, set
$$F_{r}(\Phi, \Psi)= \sup\left\lbrace \left| \int f d\Phi - \int f d\Psi \right| : f \in \mathcal{L}(r) \right\rbrace.$$
We also define $\mathcal{F}$ to be 
$$\mathcal{F}(\Phi, \Psi) = \sum_{k} 2^{-k} F_{k}(\Phi,\Psi). $$ 
It is easily seen that $F_r$ satisfies the triangle inequality for each $r>0$ and that $\mathcal{F}$ is a metric.
\end{definition}

\begin{proposition}
Let $\Phi$, $\Phi_{k}$ be Radon measures on $\RR^{d}$. Then the following are equivalent:
\begin{enumerate}
\item $\Phi_j \rightharpoonup \Phi.$
\item $\lim \mathcal{F}(\Phi_j, \Phi) \to 0$
\item For all $r>0$, $\lim_{j \to \infty} F_{r}(\Phi_{j}, \Phi)=0.$

\end{enumerate}

\end{proposition}

Let $\mu$ be a Radon measure on $\RR^{d}$ and $\Sigma$ its support.
For $a \in \RR^{d}$, $r>0$, define $T_{a,r}$ to be the following homothety that blows up $B_r(a)$  to $B_1(0)$:
\begin{equation} 
T_{a,r}(x)=\frac{x-a}{r} . \nonumber
\end{equation}
We define the image $T_{a,r}[\mu] $ of $\mu$ under $T_{a,r}$ to be the following measure:
\begin{align*}
T_{a,r}[\mu](A) & =\mu(T_{a,r}^{-1}(A)),\\ &=\mu(rA+a), \text{  } A \subset {\RR}^{d}. 
\end{align*}

\begin{definition}[\cite{P}]\label{tangent*}
We say that $\nu$ is a tangent measure of $\mu$ at a point $x_0 \in \RR^d$ if $\nu$ is a non-zero Radon measure on $\RR^n$ and if there exist sequences $(r_i)$ and $(c_i)$ of positive numbers such that $r_i \downarrow 0$ and:
\begin{equation} \label{tangent'}
c_i T_{x_0,r_i}[\mu]\rightharpoonup \nu \text{ as } i \rightarrow \infty, 
\end{equation}
where the convergence in ($\ref{tangent'}$) is the weak convergence of measures.
We write $\nu \in \mbox{Tan}(\mu,x_0)$.
 \end{definition}

\begin{remark}
 By Remark 14.4 in $\cite{M}$, if 
 \begin{equation} \label{upperlowerdensity}
 \limsup_{r \to 0} \frac{\mu(B_{2r}(x_0))}{\mu(B_r(x_0))} < \infty
 \end{equation}
 and if $\nu \in \text{Tan}(\mu,x_0)$, then we can choose ($r_i$) such that:
 \begin{equation} \label{tangent}
\mu(B_{r_i}(x_i))^{-1} T_{x_0,r_i}[\mu]\rightharpoonup c \nu \text{ as } i \rightarrow \infty,
 \end{equation}
for some $c>0$. 
We denote $\mu(B_{r}(x))^{-1} T_{x,r}[\mu]$ by $\mu_{x,r}$.
\end{remark}
The more general notion of pseudo-tangents was introduced by Toro and Kenig in $\cite{KT}$.
\begin{definition} 
Let $\mu$ be a doubling Radon measure in $\RR^{d}$. We say that $\nu$ is a pseudo-tangent measure of $\mu$ at the point $x \in \mbox{supp} \mu$ if $\nu$ is a nonzero Radon measure in $\RR^{d}$ and if there exists a sequence of points $x_i \in \mbox{supp} \mu$  such that $x_i \to x$ and  sequences of positive numbers $\left\lbrace r_i \right\rbrace$ and $\left\lbrace c_i \right\rbrace$ such that $r_i \downarrow 0$ and $c_i T_{{x_i},{r_i}} \sharp \mu \rightharpoonup \nu$.
\end{definition} 

\begin{definition}\label{flat}
A measure  on $\RR^d$ is called $n$-flat if it is equal to $ c \mathcal{H}^{n} \res V$, where $V$ is an $n$-plane, and $0<c<\infty$.

Let $\mu$ be a Radon measure on $\RR^d$ and $x_0$  be a point in the support $\Sigma$ of $\mu$. We will call $x_0$ a flat (or regular) point of $\Sigma$ if there exists an $n$-plane $V$ such that
\begin{equation}\label{flatpoint}
\text{Tan}(\mu,x_0)= \left\lbrace c \mathcal{H}^{n} \res V \ ; \ c >0 \right\rbrace. 
\end{equation}

Any point of $\Sigma$ that is not flat will be called a singular (or non-flat) point.

\end{definition}

\begin{definition}
Let $\mu$ be a Radon measure in $\RR^d$.
\begin{itemize}
\item We say $\mu$ is uniformly distributed or uniform if there exists a positive function $\phi : \RR_{+} \rightarrow \RR_{+}$ such that:
$$
\mu(B_r(x))=\phi(r), \text{ for all } x \in \Sigma, r>0.$$
We call $\phi$ the distribution function of $\mu$.

\item If there exists $c>0$ such that $\phi(r)=c r^n$, we say $\mu$ is $n$-uniform.

\item If $\mu$ is an $n$-uniform measure such that $T_{0,r}[\mu] = r^{n} \mu $ for all $r>0$, we call it a conical $n$-uniform measure.

\end{itemize}

\end{definition}

In [$\cite{P}$, Theorem $3.11$], Preiss showed that if $\mu$ is a uniform measure, there exists a unique $p$-uniform measure $\lambda$ such that:
\begin{equation}
r^{-n} T_{x,r} [\mu] \rightharpoonup \lambda, \mbox{ as } r \to \infty, 
\end{equation}
for all $x \in \RR^{d}$. $\lambda$ is called the tangent measure of $\mu$ at $\infty$.

\begin{theorem}[\cite{P}, Theorem 3.11] \label{dim0} Suppose $\mu$ is a uniform measure in $\mathbb{R}^d$, and let $f$ be its distribution function. Then there exist integers $n$ and $p$ such that: $$ \lim_{r \to 0} \frac{f(r)}{r^n} \mbox{ and } \lim_{r \to \infty} \frac{f(r)}{r^p} \mbox{  both exist and are in } (0,\infty)$$ 
We denote $n$ and $p$ by $$n=dim_0 \mu \mbox{ and } p=dim_{\infty} \mu.$$ \end{theorem}

\begin{theorem}[3.11, \cite{P}]\label{Preiss}
 Let $\mu$ be a uniform measure in $\RR^{d}$. Then, for every $x \in \Sigma \cup \left\lbrace \infty \right\rbrace$, there exist integers $n=dim_{0} \mu$ and $p=dim_{\infty} \mu$ , a unique conical $n$-uniform measure ${\lambda}_x$ and a unique conical $p$-uniform measure such that:
 \begin{itemize}
 \item Tan($\mu$,$x$)=$\left\lbrace c \lambda_{x} ; c>0 \right\rbrace$
 \item $\lim_{r \to 0} \mathcal{F}(\mu_{x,r} , {\lambda}_{x})= 0$ if $x \neq \infty$.
 \item $\lim_{r \to \infty} \mathcal{F}(\mu_{y,r} , \lambda_{\infty})=0$ for each $y \in \RR^{d}$.
\end{itemize}
 Moreover, for $\mu$-almost every $x \in \Sigma$, ${\lambda}_{x}$ is flat.

\end{theorem}

\begin{definition}\label{normalized} Let $\mu$ be a uniform measure in $\RR^{d}$, $x_0 \in \mbox{supp}(\mu) \cup \left\lbrace \infty \right\rbrace$. We will call $\mu^{x_0}$ the normalized tangent measure to $\mu$ at $x_0$ if $\mu^{x_0} \in \mbox{Tan}(\mu,x_0)$, and $\mu^{x_0}(B_1(0))=\omega_n$.
 \end{definition}
 
One of the most remarkable results in Preiss' paper $\cite{P}$ is a separation between flat and non-flat measures at infinity. We will state a reformulation of this theorem by De Lellis from $\cite{Del}$ which is better adapted to our needs.

\begin{theorem}[\cite{P}]\label{flatnessinfty}
Let $\mu$ be an  $n$-uniform measure in $\RR^{d}$, $\zeta$ its normalized tangent at $\infty$  (in the sense of Definition $\eqref{normalized}$). If $n \geq 3$, then there exists $\epsilon_{0}>0$ (depending only on $n$ and $d$) such that, if \begin{equation}
\min_{V \in G(n,d)} \int_{B_1(0)} dist^{2}(z,V) d\zeta(z) \leq \epsilon_{0},
\end{equation}
then $\mu$ is flat.

In particular, if $\mu$ is conical and \begin{equation} \min_{V \in G(n,d)} \int_{B_1(0)} dist^{2}(z,V) d\mu(z) \leq \epsilon_{0},
\end{equation} then $\mu$ is flat.
\end{theorem}

$\cite{Del}$  defines certain functionals that measure how far from flat a measure is and behave well under weak convergence.

\begin{definition} \label{functional}
Let $\varphi \in C_{c}(B_1(0))$, $0 \leq \varphi \leq 1$ and $\varphi=1$ on $B_1(0)$.
We define the functional $F: \mathcal{M}(\RR^{d}) \to \RR$ as
$$ F(\Phi): = \min_{V \in G(n,d)} \frac{1}{\Phi(B_1(0)}\int \varphi(z) {dist}^{2}(z,V)d\Phi(z)$$
\end{definition}

\begin{lemma}[\cite{Del}]\label{contfunctional}
Let $\Phi_{j}$ , $\Phi$ be Radon measures such that $\Phi_{j} \rightharpoonup \Phi$.
Then $F(\Phi_{j}) \to F(\Phi)$. 
\end{lemma}

Theorem $\ref{flatnessinfty}$ is easily reformulated in terms of the functionals $F$ in the following way.

\begin{corollary}\label{functionalflatness}
Let $\mu$ be a uniform measure on $\RR^{d}$. If $n \geq 3$, there exists $\epsilon_{0}>0$ (depending only on $n$ and $d$) such that 
$$\limsup_{R \to \infty} F(\mu_{0,R}) \leq \epsilon_{0} \implies \mu \mbox{ is flat}.  $$
In particular, if $\mu$ is conical and $F(\mu) \leq \epsilon_{0}$ then $\mu$ is flat.
\end{corollary}

Let us define the notion of asymptotically optimally doubling measures.
\begin{definition} If $x \in \Sigma$, $r>0$ and $t \in (0,1]$, define the quantity:
\begin{equation}
R_{t}(x,r)= \frac{\mu(B_{tr}(x))}{\mu(B_r(x))} - t^{n}.
\end{equation}
We say $\mu$ is asymptotically optimally doubling if for each compact set $K \subset \Sigma$, $x \in K$, and $t \in [\frac{1}{2},1]$ 
\begin{equation}\label{asymptoptim}
\lim_{r \to 0^{+}} \sup_{x \in K} \left| R_{t}(x,r) \right| = 0.
\end{equation}

\end{definition}
 The  following theorem from $\cite{KiP}$ states that uniformly distributed measures don't grow too fast.

\begin{theorem}[\cite{KiP}, Lemma 1.1]\label{doubling}
Let $\mu$ be a uniformly distributed measure over $\RR^{d}$, $x \in \RR^{d}$, $0<s<r<\infty$ and $\phi$ its distribution function. Then $\mu(B_r(x)) \leq 5^{d} \left( \frac{r}{s} \right)^{d} \phi(s)$.
\end{theorem}

\newpage

\section{Pseudo-tangents of Uniformly Asymptotically Doubling measures}
We first introduce the notion of a uniformly asymptotically doubling measure.
\begin{definition}
Let $\mu$ be a Radon doubling measure in $\mathbb{R}^{d}$, $\Sigma=spt(\mu)$. We say $\mu$ is uniformly asymptotically doubling (UAD) if  there exists a continuous function $f_{\mu}: \Sigma \times \RR_{+} \to \mathbb{R}_+$, $f_{\mu}(x,1)=1$ for every $x \in \Sigma$ such that, for every $K$ compact with  $K \cap \Sigma \neq \emptyset$, and for every $\epsilon>0$, there exists $r_{K}>0$ such that:

\begin{equation}
r \leq r_{K} \implies \left| \frac{\mu(B_{tr}(x))}{\mu(B_r(x)} - f_{\mu}(x,t) \right| < \epsilon, \mbox{ for every } x \in K \cap \Sigma, \; t \in (0,1].
\end{equation}

We will denote $f_{\mu}$ by $f$ when there is no ambiguity in doing so and call $f_{\mu}$ the distribution function associated to $\mu$. We also denote $\frac{\mu(B_{tr}(x))}{\mu(B_r(x))}$ by $R(x,r,t)$ and call it the doubling ratio of $\mu$.
\end{definition}

The proof of the following two lemmas is similar to the proofs of Lemma [2.1] and Theorem [2.1] in $\cite{KT}$.
\begin{lemma}
Let $\mu$ be a uniformly asymptotically doubling measure in $\mathbb{R}^d$. Assume $\xi_i \to \xi$, $r_i \to 0$ and $\mu_{\xi_i, r_i} \rightharpoonup \nu$. If $\Sigma=spt(\mu)$ then 
$$x \in spt(\nu) \iff \mbox{ there exists a sequence } x_i \in \frac{\Sigma-\xi_{i}}{r_i} \; \mbox{ such that } x_i \to x.$$
\end{lemma}

\begin{proof}

We first prove that if $x_i \to x$ where $x_i=\frac{z_i - \xi_i}{r_i}$ for $z_i \in \Sigma$, $r\in (0,1)$ then $\nu(B_r(x))>0$.
Let $i_0$ be such that:

$$ i \geq i_0 \implies |x-x_i|<\frac{r}{2} \;, \; |z_i - \xi_i| \leq r_i |x_i| \leq r_i M, \mbox{ where } M=|x|+1.$$

Let $t= \frac{r}{2(M+1)}$.

Since $\mu$ is UAD, there exists $R>0$ such that for $y \in B_1(\xi)$, $0<r<R$, $$\frac{1}{2} f_{\mu}(y,t) \leq \frac{\mu(B_{rt}(y))}{\mu(B_r(y))} \leq 2  f_{\mu}(y,t). $$ 
Since $\xi_i \to \xi$, $r_i \to 0$ and $|z_i-\xi_i| \leq Mr_i$, there exists $i_1 \geq i_0$ such that 
$$i \geq i_1 \implies z_i \in B_1(\xi)\;,\; \frac{r r_i}{2} \leq R \mbox{ and } (M+1)r_i \leq R.$$

We get: \begin{align*}
\mu_{\xi_i, r_i}(B_r(x)) & = \frac{\mu(B_{rr_i}(\xi_i+r_i x))}{\mu(B_{r_i}(\xi_i))}, \\ & \geq \frac{\mu(B_{rr_i - r_i|x-x_i|}(z_i)}{\mu(B_{r_i}(\xi_i))}, \\& \geq \frac{\mu(B_{\frac{r r_i}{2}}(z_i) )}{\mu(B_{r_i}(\xi_i))}, \\& \geq \frac{\mu(B_{\frac{r r_i}{2}}(z_i) ))}{\mu(B_{r_i (M+1)}(z_i) ))}, \\& \geq \frac{1}{2} f_{\mu}(z_i,\frac{r}{2(M+1)}).
\end{align*}

Therefore \begin{align*}
\nu(B_{2r}(x)) & \geq \nu(\overline{B}_r(x)), \\ &  \geq \limsup\mu_{\xi_i, r_i}(B_r(x)) , \\ & \geq \limsup\frac{1}{2} f_{\mu}(z_i,\frac{r}{2(M+1)}), \\& \geq \frac{1}{2} f_{\mu}(\xi,\frac{r}{2(M+1)}), \\ & >0.
\end{align*}

To prove the converse, suppose that $x \in spt(\nu)$ and that there exists a subsequence $i_k$ such that $$d(x, \frac{\Sigma-\xi_{i_k}}{r_{i_k}})\geq \epsilon_0.$$
Then $$B_{\frac{\epsilon_0}{2}}(x) \cap \frac{\Sigma-\xi_{i_k}}{r_{i_k}} = \emptyset.$$

Take $\phi \in C_{c}(B_{\frac{\epsilon_0}{2}}(x))$. Then

$$ \int \phi d\nu= \lim \frac{1}{\mu(B_{r_{i_k}}(\xi_{i_k}))} \int \phi( \frac{y-\xi_{i_k}}{r_{i_k}})d\mu(y) = 0.$$
This contradicts $x \in spt(\nu)$.
\end{proof}
We restate Theorem $\ref{pseudotangentUAD}$ before proving it.
\begin{theorem} 
Let $\mu$ be a uniformly asymptotically doubling measure in $\RR^{d}$. Then all pseudo-tangents of $\mu$  are uniform. More precisely, if $\xi \in supp(\mu)$, and $\nu$ is a pseudo-tangent to $\mu$ at $\xi$, then for every $x \in supp(\nu)$, and every $r>0$ we have :
$$\nu(B_{r}(x))=f_{\mu}(\xi,r).$$
\end{theorem}

\begin{proof}
Suppose $\xi \in \Sigma$, $\nu$ a pseudo-tangent to $\mu$ at $\xi$.

We estimate $\nu(B_r(x))$ for $x \in spt(\nu)$, $r>0$. Fix $\epsilon>0$.

Let $\xi_i$ be a sequence of points in $\Sigma$ such that $\xi_i \to \xi$ and $r_i$ a sequence of positive radii decreasing to $0$ chosen so that $$\mu_{\xi_i, r_i} \rightharpoonup \nu.$$ Let $x_i$ a sequence of points in $\frac{\Sigma-\xi_i}{r_i}$ converging to $x$ and let $z_i \in \Sigma$ be such that $z_i=r_i x_i + \xi_i$.

Choose $i_0$ so that: 
$$i \geq i_0 \implies |x-x_i|<\min(1,\epsilon r) \; ; \;  |z_i-\xi_i|=r_i|x_i| \leq M r_i \; \mbox{ where } M=|x|+1.$$

Then: \begin{align*}
\mu_{\xi_i, r_i}(B_r(x)) &= \frac{\mu(B_{rr_i}(\xi_i+r_ix))}{\mu(B_{r_i}(\xi_i))},\\ & \geq \frac{\mu(B_{(1-\epsilon)r_i r}(z_i))}{\mu(B_{r_i}(\xi_i))}, \\ &= \frac{\mu(B_{(1-\epsilon)r_i r}(z_i)) }{\mu(B_{(1-\epsilon)r_i r}(\xi_i))} . \frac{\mu(B_{(1-\epsilon)r_i r}(\xi_i))}{\mu(B_{r_i}(\xi_i)}.
\end{align*}

Choosing $i$ large enough, we have on one hand \begin{equation} \label{ineq1} \frac{\mu(B_{(1-\epsilon)r_i r}(\xi_i))}{\mu(B_{r_i}(\xi_i)} \geq (1- \epsilon) f_{\mu}(\xi_i, (1-\epsilon)r),
\end{equation}

and on the other hand, for $\kappa>1$, we have:

\begin{align}\label{ineq2}
\frac{\mu(B_{(1-\epsilon)r_i} r (z_i))}{\mu(B_{(1-\epsilon)r_i r} (\xi_i))} & \geq (1-\epsilon)^2 \frac{\mu(B_{\kappa(1-\epsilon)r_i r} (z_i))}{\mu(B_{\kappa(1-\epsilon)r_i r} (\xi_i))}, \nonumber \\ & \geq (1-\epsilon)^2 \frac{\mu(B_{\kappa(1-\epsilon)r_i r - Mr_i} (\xi_i))}{\mu(B_{\kappa(1-\epsilon)r_i r} (\xi_i))}, \\ & \geq (1-\epsilon)^{3} f_{\mu}(\xi_i, 1-\epsilon-\frac{M}{\kappa r}). \nonumber
\end{align}
Let $\kappa_{\epsilon}$ be chosen so that $\frac{M}{\kappa_{\epsilon} r}< \epsilon$.

Putting $\ref{ineq1}$ and $\ref{ineq2}$ together, we get:

\begin{equation}\label{ineq3}
\mu_{\xi_i,r_i}(B_r(x)) \geq (1-\epsilon)^{4} f_{\mu}(\xi_i, (1-\epsilon)r) f_{\mu}(\xi_i, 1-\epsilon-\frac{M}{\kappa_{\epsilon} r} ).
\end{equation}

Letting $i \to \infty$, we get 
\begin{align*}
 \liminf \mu_{\xi_i,r_i}(B_r(x)) & \geq   \liminf  f_{\mu}(\xi_i,(1-\epsilon)r) \; . \; \liminf  f_{\mu}(\xi_i, 1-\epsilon-\frac{M}{\kappa r}), \\
 &=  f_{\mu}(\xi, (1-\epsilon)r) \; . \; \liminf  f_{\mu}(\xi, 1-\epsilon-\frac{M}{\kappa r}).
\end{align*}

Letting $\epsilon$ go to $0$, we get:

$$\liminf \mu_{\xi_i,r_i}(B_r(x)) \leq f_{\mu}(\xi, r).$$

A similar calculation gives: $$\limsup \mu_{\xi_i,r_i}(B_r(x))  \leq f_{\mu}(\xi,r).$$

On one hand we get $$ \nu(B_r(x)) \leq \liminf \mu_{\xi_i,r_i}(B_r(x))  \leq  f_{\mu}(\xi,r)$$. 

On the other hand, for $\delta>0$ chosen arbitrarily we get:

$$\nu(B_r(x)) \geq \nu(\overline{B}_{(1-\delta)r}(x)) \geq \limsup \mu_{\xi_i,r_i}(\overline{B}_{(1-\delta)r}(x))  \geq  f_{\mu}(\xi,(1-\delta)r).$$

Letting $\delta$ go to $0$, we obtain:

$$\nu(B_r(x))=f_{\mu}(\xi,r).$$

\end{proof}









\begin{corollary}
Let $\mu$ be a Uniformly Asymptotically Doubling measure and $f$ be its distribution function. Then for every $x$ there exists $n=n_{x}$ such that: $$\lim_{t \to 0} \frac{f(x,t)}{t^n}=f(x),$$ where $f(x) \in (0,\infty)$. We write $n_x= \dim f(x,.)$.
\end{corollary}

\begin{proof}
This is a direct consequence of the fact that for fixed $x$, $f(x,t)$ is the distribution function of a uniform measure and of Theorem $\ref{dim0}$.
\end{proof}

\begin{definition}
Let $\mu$ be a UAD measure in $\RR^{d}$ with distribution function $f_{\mu}(x,t)$. Let $n=\sup \left\lbrace \dim f(x,.) \; ; \; x \in supp(\mu) \right\rbrace $(Note that $n \leq d$). We say $\mu$ is $n$-UAD.
\end{definition}
\newpage

\section{Singularities of UAD measures}

We first aim to prove Theorem $\ref{accumsing}$. We start by proving an analogue of Lemma [1.5] from $\cite{N1}$ stating a ``connectedness'' result for pseudo-blow ups along the same sequence of points.
\begin{lemma} \label{connectedness}
Let $\mu$ be a uniformly asymptotically doubling measure with $\left\lbrace x_k \right\rbrace_{k=1}^{\infty} \subset supp(\mu) \cap \overline{B}_1(0)$, $x_k \to x$. Moreover, let $\left\lbrace \tau_k \right\rbrace$, $ \left\lbrace \sigma_k \right\rbrace $ sequences of positive numbers going to zero. We also assume that $\sigma_{k} < \tau_{k}$ and there exist uniform measures $\alpha$ and $\beta$ such that: 
$$\mu_{x_k, \tau_k} \rightharpoonup \alpha \mbox{ and } \mu_{x_k, \sigma_k} \rightharpoonup \beta. $$
Then 
$$ \alpha \mbox{ is flat } \implies \beta \mbox{ is flat.} $$
\end{lemma}

\begin{proof}
Assume that $\alpha$ is flat but $\beta$ is not. Then $F(\alpha)=0$ and there exists $R_0>0$ such that :$$r \geq R_0 \implies F(\beta_{0,r})>\epsilon_{0}.$$
By continuity of $F$, there exists $0<\kappa<\epsilon_0$ and $k_0$ so that: $$k>k_0 \implies F(\mu_{x_k, \tau_k})< \kappa \; \mbox{ and} \; F(\mu_{x_k, R_0\sigma_k})>\kappa.$$

We claim that we can assume without loss of generality that $R_0 \sigma_{k} < \tau_{k}$. In fact we prove that $lim_{k \to \infty} \frac{\tau_k}{\sigma_{k}}= \infty$.

Indeed, assume by contradiction and passing to a subsequence that $\frac{\tau_k}{\sigma_k} \to \gamma>1$. Let $\gamma_k=\frac{\tau_k}{\sigma_k}$.
Then, on one hand, 
\begin{align*}
(\mu(B_{\sigma_k}(x_k)))^{-1} T_{x_k, \tau_k}[ \mu] &= (\mu(B_{\sigma_k}(x_k)))^{-1} T_{x_k, \gamma_k \sigma_k} [ \mu], \\& \rightharpoonup T_{0,\gamma}[ \beta].
\end{align*}

On the other hand, for $k$ large enough,
\begin{align*}
\frac{\mu(B_{\tau_k}(x_k))}{\mu(B_{\sigma_k}(x_k))}&=\frac{\mu(B_{\gamma_k \sigma_k}(x_k))}{\mu(B_{\sigma_k }(x_k))} \\  &\leq 2 f(x_k, \gamma_k)^{-1}
\end{align*}

so that 
\begin{align*}
\limsup_{k} \frac{\mu(B_{\tau_k}(x_k))}{\mu(B_{\sigma_k}(x_k))} &\leq 2 \limsup_k f(x_k, \gamma_k)^{-1} \\ & = 2 f(x, \gamma).
\end{align*}

Passing to a subsequence we get:

$$(\mu(B_{\sigma_k}(x_k)))^{-1} T_{x_k, \tau_k} [ \mu ]= \frac{\mu(B_{\tau_k}(x_k))}{\mu(B_{\sigma_k}(x_k))} d\mu_{x_k, \tau_k} \rightharpoonup c \alpha,$$
for some constant $c$.

However $\beta$ is not flat and $\alpha$ is flat thus yielding a contradiction.

Now define $f_k:(0,\infty) \to (0,\infty)$ to be $$f_k(r)= F(\mu_{x_k,r}).$$
$f_k$ is continuous away from zero for every $k>0$.
In particular, for every $k>k_0$, there exists $\delta_k \in [R_0 \sigma_k,\tau_{k}]$ so that $$F(\mu_{x_k, \delta_k})=\kappa \; \mbox{ and } F(\mu_{x_k,r}) \leq \kappa \mbox{ for } r \in [\delta_k, \tau_k].$$
Without loss of generality, by passing to a subsequence, 
$$\mu_{x_k,\delta_k} \to \xi.$$
Moreover, since $\xi$ is a pseudo-tangent to $\mu$ at $x$, it is in particular uniform.
By continuity of $F$, $F(\xi)=\kappa$. In particular, $\xi$ is not flat. We now show that $\xi$ is flat at infinity, obtaining hence a contradiction.

In the same way that we proved that $\frac{\tau_k}{\sigma_k} \to \infty$, we can prove that $\frac{\tau_k}{\delta_k} \to \infty$.
Now fix $R>1$. Since $\frac{\tau_k}{\delta_k} \to \infty$, there exists $k_1>k_0$ such that :
$$k>k_1 \implies R\delta_k \in [\delta_k, \tau_k].$$
In particular, if $k>k_1$, $F(\mu_{x_k,R\delta_k}) \leq \kappa$.
We deduce that $$\limsup F(\mu_{x_k,R\delta_k}) \leq \kappa.$$

Using the fact that $\frac{(\mu(B_{\delta_k}(x_k))}{\mu(B_{R\delta_k}(x_k))} \to f(x,\frac{1}{R})^{-1}=c_R$ , we have that for all $s>0$,
$$\lim_{k \to \infty} F_{s} ( \mu_{x_k, R\delta_k}, c_R T_{0,R} [ \xi])=0,$$
and hence, $$ \mu_{x_k, R\delta_k}\to c_R T_{0,R} [ \xi].$$

It follows easily from the fact that $ \mu_{x_k, R\delta_k}(B_1(0))=1$ that $c_R T_{0,R} [ \xi] = \xi_{0,R}$.

Thus, we have $$\mu_{x_k,R\delta_k} \rightharpoonup \xi_{0,R},$$

and consequently, $F(\xi_{0,R}) \leq \kappa$ for all $R>1$.

Now letting $R \to \infty$, we get $F(\Psi) \leq \kappa < \epsilon_0$, where $\Psi$ is the tangent to $\xi$ at $\infty$. This implies that $\Psi$ is flat and consequently, $\xi$ is flat contradicting the fact that $F(\xi)=\kappa$.

\end{proof}
Let us restate Theorem $\ref{accumsing}$ before proving it.
\begin{theorem} 
Let $\mu$ be a UAD measure.
Let $\left\lbrace x_j \right\rbrace_{j=0}^{\infty} \subset \mathcal{S}_{\mu}$, $x_j \to x_0$, $\nu \in Tan(\mu,x_0)$, $\left\lbrace r_j \right\rbrace_{j}$ a sequence going to zero such that $\mu_{x_0,r_j} \rightharpoonup \nu$. Moreover, let $y_j= \frac{x_j-x_0}{r_j} \in B_1(0)$, $y_j \to y$.
Then $y \in \mathcal{S}_{\nu}$.
\end{theorem}

\begin{proof}
We start by constructing a sequence $\sigma_k$ satisfying $\mu_{x_k, \sigma_k} \rightharpoonup \nu^{\infty}$, where $\nu^{\infty}$ satisfies $F(\nu_{0,R}^{\infty})>\epsilon_0$ for all $R>1$ and $\frac{\sigma_j}{r_j} \to 0$ as $j \to \infty$.

First let $\left\lbrace s_j^k \right\rbrace$ be subsequences of $r_j$ such that $\left\lbrace s_{j}^{k+1} \right\rbrace$ subsequence of $\left\lbrace s_{j}^{k} \right\rbrace$, and $\mu_{x_k, s_j^k} \rightharpoonup \nu^k$ where $\nu^{k} \in Tan(\mu,x_k)$. Moreover, for every $y \in supp(\nu^{k})$, $r>0$, we have :
$$\nu^{k}(B_r(y))=f_{\mu}(x_k,r).$$

Since $x_k$ are singular points, $\nu^{k}$ is not flat. In particular, it is curved at $\infty$, i.e. there exist $R_k$ such that whenever $r>R_k$, $F(\nu^{k}_{0,r})> \epsilon_{0}$. 

Now for all $k$, choose $s_{j_k}^{k}$ such that: $R_k s_{j_k}^{k} 
\leq r_k^2$, and $\mathcal{F}( \mu_{x_k, R_k s_{j_k}^{k}}, \nu_{0,R_k}^k)< \frac{1}{2^k}$. By passing to a subsequence if necessary, we can assume that $\nu^{k}_{0,R_k} \rightharpoonup \nu^{\infty}$ for some Radon measure $\nu^{\infty}$. Indeed, this follows from the fact that for $S>1$ fixed and for $k$ large enough, we have, by Theorem $\ref{doubling}$:
$$ \frac{\nu^{k}(B_{R_kS}(0))}{\nu^{k}(B_{R_k}(0))} \leq 5^{d} S^{d},$$
from which convergence of a subsequence follows by compactness.
 
Moreover, $\mu_{x_k, R_k s_{j_k}^{k}} \to \nu^{\infty}$ since we have $\mathcal{F}( \mu_{x_k, R_k s_{j_k}^{k}}, \nu_{0,R_k}^k)< \frac{1}{2^k}$ and $\mathcal{F}(\nu_{0,R_k}^k, \nu^{\infty})$ going to $0$ as $k$ goes to $\infty$. In particular, $\nu^{\infty}$  is a pseudo-tangent to $\mu$ at $x_0$.

We claim that for all but at most countably many $R \in (1,\infty)$, $F(\nu^{\infty}_{0,R})> \frac{\epsilon_0}{2}$.
Let $V$ be an $n$-plane such that $F(\nu_{0,R}^{\infty})= \int \phi(z) dist(z,V)^{2} d\nu_{0,R}^{\infty}(z)$. Suppose for contradiction that $F(\nu_{0,R}^{\infty}) < \frac{\epsilon_{0}}{2}$ for some $R>1$, $R \notin S_0$.

\begin{align*}
F(\nu^{k}_{0,RR_k}) & \leq \int \phi(z) dist(z,V)^{2} d\nu^{k}_{0,RR_k},\\ & \leq \int \phi (z) dist(z,V)^{2} d\nu^{k}_{0,RR_k} -\int \phi(z) dist(z,V)^{2} d\nu^{\infty}_{0,R}+ F(\nu^{\infty}_{0,R}), \\ & \leq F_1(\nu^{\infty}_{0,R},\nu^{k}_{0,RR_k})+\frac{\epsilon_0}{2}
\end{align*} 

To prove that the right hand side goes to $0$ as $k$ goes to $\infty$, let $g \in \mathcal{L}(1)$. Then, if we define  $g_R(y)= R g (\frac{y}{R})$, we have:

\begin{align*}
\left|\int g d\nu^{k}_{0,RR_k} - \int g d\nu^{\infty}_{0,R} \right|& = \left| \frac{\nu^{k}_{0,R_k}(B_{1}(0))}{R\nu^{k}_{0,R_k}(B_{R}(0))} \int g_R d\nu^{k}_{0,R_k}- \frac{1}{R\nu^{\infty}(B_R(0)} \int g_R d\nu^{\infty} \right|, \\ & \leq \frac{\nu^{k}_{0,R_k}(B_{1}(0))}{R\nu^{k}_{0,R_k}(B_{R}(0))} \left| \int g_R d\nu^{k}_{0,R_k}- \int g_R d\nu^{\infty} \right| \\ &+\left|\frac{\nu^{k}_{0,R_k}(B_{1}(0))}{R\nu^{k}_{0,R_k}(B_{R}(0))} - \frac{1}{R\nu^{\infty}(B_R(0))} \right|\int g_R d\nu^{\infty}, \\ & \leq \frac{1}{R\nu^{k}_{0,R_k}(B_{R}(0))} F_R(\nu^{k}_{0,R_k}, \nu^{\infty}) \\ &+\left|\frac{1}{\nu^{k}_{0,R_k}(B_{R}(0))} - \frac{1}{\nu^{\infty}(B_R(0))}\right| 2.\nu^{\infty}(B_R(0)).
\end{align*}
However since $\nu^{k}_{0,R_k} \rightharpoonup \nu^{\infty}$, we have $\nu^{k}_{0,R_k}(B_R(0)) \to \nu^{\infty}(B_R(0))$ for all but countably many $R>1$ and $F_R(\nu^{k}_{0,R_k}, \nu^{\infty}) \to 0$.

Therefore, choosing $k$ large enough, we get $F(\nu^{k}_{0,RR_k}) \leq \epsilon_0$
This contradicts the definition of $R_k$. We therefore have $F(\nu_{0,R}^{\infty})> \frac{\epsilon_0}{2}$ for all but at most countably many $R>1$.
Letting $R$ to $\infty$, we get that $\nu^{\infty}$ is curved at $\infty$. However since $\nu^{\infty}$ is a pseudo-tangent to $\mu$, it is in particular uniform which implies that it is not flat.
Letting $\sigma_k=R_k s_{j_k}^{k}$, our claim is proved.

On the other hand, we claim that 
\begin{equation} \label{nuy}
y \in supp(\nu) \mbox{ and } \mu_{x_j , r_j} \rightharpoonup \nu_{y}. 
\end{equation} where  $\nu_{z}$ denotes $T_{z,1}[\nu]$ whenever $z \in supp(\nu)$. 
 Indeed, let $\delta > 0$.
Then:
\begin{align*}
\nu(B_{\delta}(y)) & \geq \limsup_{i \to \infty} \mu_{0,r_i}\left(_{B_{\frac{\delta}{4}}}\left(y\right)\right) \\ & =\limsup_{i \to \infty} \omega_{n}(\mu(B_{r_i}(0)))^{-1} \mu\left(B_{ \frac{r_i \delta}{4}}\left(r_i y\right)\right).
\end{align*}
But for $i$ large enough $|y-y_i| \leq \frac{\delta}{8}$ implying that $B_{r_i \frac{\delta}{8}}(x_{i}) \subset B_{ \frac{r_i \delta}{4}}(r_i y)$. Consequently, $$\nu(B_{\delta}(y)) >0$$ since $$(\mu(B_{r_i}(0)))^{-1} \mu\left(B_{r_i \frac{\delta}{8}}(x_{i})\right) = \frac{\delta^{n}}{8^n}.$$

Let us prove the second part of $\eqref{nuy}$.
Recall Definition $\ref{L(r)}$.

Fix $R>0$. Let $\phi \in \mathcal{L}(R)$. Then, on one hand, for $j$ large enough that $|y_j| \leq 2$, we have:
\begin{align}
\left| \int \phi(z) d\mu_{x_j , r_j}(z) - \int \phi(z) dT_{y_j ,1}[\nu](z) \right| 
& =\left| \int \phi( z-y_j) d\mu_{0,r_j}(z) - \int \phi(z-y_j) d\nu(z) \right|, \nonumber \\
&\leq F_{R+2} (\mu_{0,r_j}, \nu) , 
\end{align}
since $\phi_j(z)=\phi(z-y_j) \in \mathcal{L}(R+2)$ .
On the other hand,
\begin{align}
\left| \int \phi(z) dT_{y_j ,1}[\nu](z)- \int \phi(z) dT_{y,1}[\nu](z) \right| &= \left| \int \left(\phi(z-y_j) - \phi(z-y)\right) d\nu(z) \right| , \nonumber \\
&\leq |y-y_j| \nu(B_{R+2}(0)),
\end{align}
since $Lip(\phi) \leq 1$, $\phi_{j}$ and $\phi_{y}$ are supported in $B_{R+2}(0)$ where we define $\phi_{y}(z)=\phi(z-y)$.
This gives, taking the supremum over all $\phi \in \mathcal{L}(R)$:
$$ F_R(\mu_{x_j , r_j} , \nu_{y} ) \leq F_{R+2} (\mu_{0,r_j}, \nu)+ |y-y_j| \nu(B_{R+2}(0)), $$
for $j$ large enough.
Letting $j \to \infty$, we get $\eqref{nuy}$ since $R$ was chosen arbitrarily. 

Let $\rho_k= \frac{\sigma_k}{r_k}$. Using $\rho_k$, we construct a sequence $\tilde{\tau}_{k}$ such that:
$$\mu_{{x}_{l_k} , \tilde{\tau}_k} \rightharpoonup \alpha,$$
for some subsequence $x_{l_k}$ of $x_k$ where $\alpha$ is the normalized tangent to $\nu$ at $y$.

For every $k$ there exists $l_k>k$, $l_k > l_{k-1}$ such that whenever $l>l_k$ 
\begin{equation}\label{tau}
F_{1} (\mu_{x_l , r_l} , \nu_{y}) < \frac{1}{k} {\rho}_{k} \nu_{y}(B_{\rho_k}(0)) \mbox{ and } \rho_{l} < \rho_{k},
\end{equation} 
since $\mu_{x_l , r_l} \rightharpoonup \nu_{y}$ and $\rho_{k} \to 0$.
Let $\tilde{\tau_{k}} = r_{l_k} \rho_{k}$ and $\tilde{x}_k= x_{l_k}$.

We claim that \begin{equation}\label{alpha}
\mu_{\tilde{x}_k , \tilde{\tau}_k} \rightharpoonup \alpha.
\end{equation}
To prove the claim, fix $R>0$.
Then for $k$ large enough that $R\rho_{k} \leq 1$ $$
F_{R}(\mu_{\tilde{x}_k , \tilde{\tau}_k} , \nu_y(B_{\rho_k}(0))^{-1} T_{0,\rho_{k}} [\nu_{y}])< \frac{2}{k}.$$ 

Indeed, let $g \in \mathcal{L}(R)$.
Then
\begin{align}
&| \int g d\mu_{\tilde{x}_k, \tilde{\tau}_k}- \frac{1}{\nu_{y}(B_{\rho_k}(0))} \int g d(T_{0,\rho_k}[\nu_y])|,\\ &= | \frac{1}{\mu(B_{\tilde{\tau}_k}(\tilde{x}_k))} \int g \; d(T_{0, \rho_k}\circ T_{\tilde{x}_k, \tilde{r_{l_k}}_k}[\mu])-\frac{1}{\nu_{y}(B_{\rho_k}(0))} \int g d(T_{0,\rho_k}[\nu_y])|, \\
& = | \frac{\mu(B_{r_{l_k}}(\tilde{x}_k))}{\mu(B_{\tilde{\tau}_k}(\tilde{x}_k))} \; \int g d(T_{0,\rho_k}[\mu_{\tilde{x}_k, r_{l_k}}])-\frac{1}{\nu_{y}(B_{\rho_k}(0))} \int g d(T_{0,\rho_k}[\nu_y])|, \\ &\leq |\frac{\mu(B_{r_{l_k}}(\tilde{x}_k))}{\mu(B_{\tilde{\tau}_k}(\tilde{x}_k))} - \frac{1}{\nu_{y}(B_{\rho_k}(0))} | \int g d(T_{0,\rho_k}[ \mu_{\tilde{x}_k, r_{l_k}}]) \\& \; \; \; \; \; \; \; +\frac{1}{\nu_{y}(B_{\rho_k}(0))} \; | \int g d(T_{0,\rho_k}[\nu_{y}])- \int g d(T_{0,\rho_k} [\mu_{\tilde{x}_k, r_{l_k}}])| \nonumber
\end{align}

On one hand, by Lemma $\ref{pseudotangentUAD}$, we have 
$$\frac{\mu(B_{r_{l_k}}(\tilde{x}_k))}{\mu(B_{\tilde{\tau}_k}(\tilde{x}_k))} - \frac{1}{\nu_{y}(B_{\rho_k}(0))} = \frac{\mu(B_{r_{l_k}}(\tilde{x}_k))}{\mu(B_{\tilde{\tau}_k}(\tilde{x}_k))} - \frac{1}{f_{\mu}(x,\rho_k)},$$
which goes to $0$ as $k$ goes to infinity since $\mu$ is UAD.

On the other hand,

\begin{align*}
(\nu_{y}(B_{\rho_k}(0)))^{-1} | \int g d(T_{0,\rho_k}[\nu_{y}])- \int g d(T_{0,\rho_k} [\mu_{\tilde{x}_k, r_{l_k}}])| & = \frac{1}{\rho_k \nu_y(B_{\rho_k}(0))}F_{R\rho_k}(\nu_{y}, \mu_{\tilde{x}_k, r_{l_k}})\\&\leq \frac{1}{\rho_k \nu_y(B_{\rho_k}(0))}F_{1}(\nu_{y}, \mu_{\tilde{x}_k, r_{l_k}}), \\ &<\frac{1}{k}
\end{align*}
 by $\ref{tau}$.
 The claim is therefore proved.

But we have $F_R (\nu_y(B_{\rho_k}(0))^{-1} T_{0,\rho_{k}} [\nu_{y}], \alpha) \to 0$ by definition of $\alpha$.
Since $$F_R(\mu_{\tilde{x}_k , \tilde{\tau}_k},\alpha) \leq F_{R}(\mu_{\tilde{x}_k , \tilde{\tau}_k} , \nu_y(B_{\rho_k}(0))^{-1} T_{0,\rho_{k}} [\nu_{y}]) + F_R (\nu_y(B_{\rho_k}(0))^{-1} T_{0,\rho_{k}} [\nu_{y}], \alpha),$$ $F_R(\mu_{\tilde{x}_k , \tilde{\tau}_k},\alpha) \to 0$. 
This proves $\eqref{alpha}$.

We have therefore obtained two sequences $\tilde{\sigma}_k=\sigma_{l_k}$ and $\tau_k=\tilde{\tau}_k$ such that:

$$\tilde{\sigma}_k<\tau_k \; , \; \mu_{x_{l_k}, \tilde{\sigma}_k} \rightharpoonup \nu^{\infty} \; , \; \mu_{x_{l_k}, \tau_k} \rightharpoonup \alpha.$$

Since $ \nu^{\infty}$ is not flat, Theorem $\ref{connectedness}$ implies that $\alpha$ is not flat which ends the proof.

\end{proof}

\begin{corollary}\label{corollaryaccum}
Let $\mu$ be a uniform measure in $\RR^{d}$ such that $\dim_{0}\mu \leq 3$. Then $\left| \mathcal{S}_{\mu} \cap K \right| < \infty$, for every $K$ compact subset of $\RR^{d}$. In particular, $dim_{H}(\mathcal{S}_{\mu})=0$. Here, $|A|$ denotes the cardinality of the set $A \subset \RR^{d}$.
\end{corollary}
\begin{proof}
Assume not. Then there exists $K$ compact subset of $\RR^{d}$ such that $\left|\mathcal{S}_{\mu} \cap K\right| = \infty$ . In particular there exists a sequence of points $\left\lbrace x_j \right\rbrace _{j} \subset \mathcal{S}_{\mu} \cap K$ converging to some  $x_{\infty} \in K$. Moreover, $x_{\infty} \in supp(\mu)$ since the support of a measure is closed set.
Let $r_j = |x_j-x_{\infty}|$ and $y_j = \frac{x_{j}-x_{\infty}}{r_j}$. Then by Theorem $\ref{Preiss}$, $\mu_{x_{\infty} , r_j } \rightharpoonup \nu$, $\nu$ normalized tangent to $\mu$ at $x_{\infty}$ and by compactness, we can assume by passing to a subsequence if necessary that $y_j \to y \in \partial B_{1}(0)$. 
By $\eqref{nuy}$, $y \in supp(\nu)$. Since $y \neq 0$, $y$ must be a flat point of $supp(\nu)$ by Theorem [??] from $\cite{N1}$. This contradicts Theorem $\ref{accumsing}$.

\end{proof}

\begin{lemma}\label{transinvariant}
Let $\nu$ be an $n$-uniform conical measure in $\mathbb{R}^d$, $\xi \neq 0$, $\xi \in spt(\nu)$, and let $\lambda$ be the tangent to $\nu$ at $\xi$. Then:
$$\lambda=\mathcal{H}^{n} \res {\mathbb{R} \times A},$$
where $A \subset \mathbb{R}^{d-1}$. Moreover, $A$ is the support of an $n-1$-uniform measure.
\end{lemma}
\begin{proof}
We will first prove that \begin{equation} \label{transinvmeasure}
T_{t \xi, 1} [\lambda]=\lambda 
\end{equation}
 for any $t>0$. 

Take $t>0$.
Then, on one hand
 \begin{align}\label{nualpha}
\nu_{(1+t)\xi, s_j} &=s_{j}^{-m} T_{\xi, \frac{s_j}{1+t}}[T_{0, 1+t}[\nu]],   \nonumber   \\
&= s_{j}^{-m} (1+t)^{m} T_{\xi,\frac{s_j}{1+t}}[\nu]  , \mbox{ since } \nu \mbox{ is conical} \nonumber \\
& \rightharpoonup \lambda,
\end{align}
since the sequence $\frac{s_j}{1+t} \to 0$ and $s_{j}^{-m} (1+t)^{m} T_{\xi,\frac{s_j}{1+t}}[\nu](B_{1}(0))=\lambda(B_{1}(0))=\omega_{m}$.

On the other hand, we have
\begin{align*}
T_{(1+t)\xi, s_j}(z)& = \frac{z-(1+t)\xi}{s_j}, \\
&=\frac{z- (1+(1-s_j)t)\xi}{s_j} -t \xi, \\
&= T_{t \xi, 1} \circ T_{(1+(1-s_j)t)\xi, s_j} (z).
\end{align*}
We now prove that 
\begin{equation}\label{transinvstep}
{s_j}^{-m} T_{(1+(1-s_j)t) \xi , s_j}[\nu] \to \lambda.
\end{equation}

Let $\phi \in \mathcal{L}(R)$.
Then, for $j$ large enough so that $|1-s_j| \leq 2$ we have:
\begin{align*}
&{s_j}^{-m} \left| \int \phi(z) dT_{(1+(1-s_j) t)\xi, s_j}[\nu](z) - \int \phi(z) dT_{(1+t)\xi, s_j}[\nu](z) \right| \\& \leq 
{s_j}^{-m}  \left| \int \left( \phi(z-(1+(1-s_j) t)\xi)  - \phi(z-(1+t)\xi) \right) dT_{0, s_j}[\nu](z) \right|, \\
& \leq {s_j}^{-m} \int_{B_{R+(1+2|t|)|\xi|}(0)} |s_j||\xi| |t| dT_{0,s_j}[\nu](z) , \\
& \leq |s_j||\xi| |t|\omega_{m} (R+(1+2|t|)|\xi|)^{m}.
\end{align*}
Taking the supremum over all $\phi \in \mathcal{L}(R)$, we get: 
\begin{align}\label{transinvstep'}
A_{j} &:= F_{R}({s_j}^{-m} T_{(1+(1-s_j)t) \xi , s_j}[\nu], {s_j}^{-m} T_{(1+t) \xi , s_j}[\nu]), \nonumber \\ & \leq |s_j||\xi| |t|\omega_{m} (R+(1+2|t|)|\xi|)^{m},
\end{align}
which goes to $0$ as $j \to \infty$ since $s_{j} \to 0$.
We have 
\begin{equation}
\label{triineq} F_{R}({s_j}^{-m} T_{(1+(1-s_j)t) \xi , s_j}[\nu], \lambda) \leq A_{j} + F_{R}({s_j}^{-m} T_{(1+t) \xi , s_j}[\nu], \lambda).
\end{equation}
Since $A_j \to 0$ by $\eqref{transinvstep'}$ and, according to $\eqref{nualpha}$, $F_{R}({s_j}^{-m} T_{(1+t) \xi , s_j}[\nu], \lambda) \to 0$, by using $\eqref{triineq}$, we prove $\eqref{transinvstep}$.

This proves $\eqref{transinvmeasure}$ from which it follows that
\begin{equation} \label{transinv}
\Sigma - t \xi = \Sigma \mbox{ for } t>0.
\end{equation}
Indeed, for $t>0$,
\begin{align*}
z \in \Sigma & \iff \mbox{For all } r>0, \lambda(B_{r}(z))>0, \\
&\iff \mbox{ For all } r>0, T_{t\xi, 1} [\lambda](B_{r}(z)) >0, \\ & \iff \mbox{ For all } r>0, \lambda(B_{r}(z+t\xi))>0, \\ &\iff z \in \Sigma - t \xi .
\end{align*}
Adding $t\xi$ on both sides of $\ref{transinv}$, we see that
\begin{equation}\label{transinvfinal}
\Sigma - t \xi = \Sigma \mbox{ for } t \in \RR.
\end{equation}

Let $e_1 = \frac{\xi}{|\xi|}$ and $A= \left\lbrace x \in \Sigma ; x.{e_1}=0 \right\rbrace $.  
We claim that 
\begin{equation}\label{directproduct} \Sigma= \RR e_1 \oplus A.  
\end{equation}
On one hand, if $z \in \RR e_1 \oplus A$, then there exists $z' \in A$ and $t \in \RR$ such that:
$$z= z' + te_1.$$
Since $A\subset \Sigma$ by definition, this implies that $z \in \Sigma+ te_1$ and consequently, $z \in \Sigma$ by $\eqref{transinv}$.
On the other hand, if $z \in \Sigma$, we can write:
$$ z= (z - \left\langle z,e_1 \right\rangle e_1)+\left\langle z,e_1 \right\rangle e_1.$$
Let $t_1 =  \left\langle z,e_1 \right\rangle$. By $\eqref{transinv}$, $z-t_1e_1 \in \Sigma$. Moreover, $\left\langle z-t_1e_1 , e_1 \right\rangle =0$. Therefore, $z-t_1e_1\in A$ and $z \in \RR e_1 + A$.
The uniqueness of such a decomposition follows from the fact that $\RR e_1$ and $A$ are orthogonal by construction.
This proves $\eqref{directproduct}$.
\end{proof}

We restate Theorem $\ref{maintheorem}$ before proving it.
\begin{theorem}
 Let $\mu$ be a $n$-UAD measure in $\RR^{d}$ , $ 3 \leq n \leq d$ where $n$ is the dimension of $\mu$. Then 
 \begin{equation} \label{hdimsing}
 dim_{\mathcal{H}}(\mathcal{S}_{\mu}) \leq n-3,
\end{equation}
where $dim_{\mathcal{H}}$ denotes the Hausdorff dimension.
\end{theorem}
\begin{proof}
The proof of this theorem is similar to the proof of Theorem [] in $\ref{N1}$. We repeat it for the reader's convenience.

The theorem holds for $n=3$. Indeed, suppose $\mu$ is $3$-UAD and $s$ is such that $\mathcal{H}^{s}(\mathcal{S}_{\mu})>0$. Then by an argument which will be outlined in the next few paragraphs, there exists $x_0 \in \mathcal{S}_{\mu}$ and $\nu \in Tan(\mu,x_0)$ such that $\mathcal{H}^{s}(\mathcal{S}_{\nu} \cap \overline{B}_1(0))>0$. But $\nu$ uniform and $\dim_{0}\nu \leq 3$ implies that $s=0$ by Corollary $\ref{corollaryaccum}$.

Let $m< d$ and assume the theorem holds for all $l$-uniform measures in $\RR^{d}$ such that $l < m$. We want to prove that it holds for $m$-UAD measures.

Suppose that $s \in \RR_{+}$ is such that $\mathcal{H}^{s}(\mathcal{S}_{\mu}) >0$.

We first find  a singular point $x_0$ of the support of $\mu$ such that the following holds: if $\nu$ is a tangent to $\mu$ at $x_0$, then:
$$\mathcal{H}^{s}(\mathcal{S}_{\nu} \cap \overline{B_1(0)})>0.$$

By Lemma 4.6 in $\cite{M}$, $$\mathcal{H}^{s}(\mathcal{S}_{\mu}) >0 \iff \mathcal{H}^{s}_{\infty}(\mathcal{S}_{\mu}) >0.$$

Since $\mathcal{H}_{\infty}^{s}(\mathcal{S}_{\mu} )>0$, there exists a compact set $K$ such that $\mathcal{H}_{\infty}^{s}(\mathcal{S}_{\mu} \cap K )>0$. Let $\tilde{\mathcal{S}_{\mu}}= \mathcal{S_{\mu}} \cap K$.
We have
\begin{equation}\label{densityupper'}
\theta^{s,*}(\mathcal{H}^{s}_{\infty}\res \tilde{{\mathcal{S}_{\mu}} }, z) \geq 2^{-s},
\end{equation}
for $\mathcal{H}^{s}$-almost every $z \in  \tilde{\mathcal{S}_{\mu}}$.
This follows from Theorem 3.26 (2), in $\cite{S}$ since $\tilde{\mathcal{S}_{\mu}}$ is a compact subset of $\RR^d$.
In particular, there exists $x_0 \in \tilde{\mathcal{S}_{\mu}}$ such that:
\begin{equation}
\theta^{s,*}(\mathcal{H}^{s}_{\infty}\res \tilde{{\mathcal{S}_{\mu}} }, x_0) \geq 2^{-s},
\end{equation}

Consequently, there exists a sequence of  radii $\left\lbrace r_{j} \right\rbrace _{j}$ decreasing to $0$ such that:
$$\mathcal{H}_{\infty}^{s} \left(\overline{B_1(0)} \cap \frac{\tilde{{\mathcal{S}}_{\mu}} - x_0}{r_j}\right) \geq 2^{-s}. $$
Since $r_j \downarrow 0$, $\mu_{ x_0 , r_j } \rightharpoonup \nu$ where $\nu$is a tangent to $\mu$ at $x_0$.
By Theorem $\ref{accumsing}$, for all $\epsilon >0$, there exists $j_{0}$ such that: 
\begin{equation} \label{nhoodsing}
\frac{\tilde{\mathcal{S}_{\mu}} - x_{0}}{r_j} \cap \overline{B_1(0)} \subset \left( \mathcal{S}_{\nu} \right)_{\epsilon}  \mbox{ whenever }         j \geq j_{0}.
\end{equation}
Pick $\delta > 0$ and let $\left\lbrace E_k \right\rbrace _{k}$ be a covering of $\tilde{\mathcal{S}_{\nu}} = \mathcal{S}_{\nu} \cap \overline{B_1(0)}$ such that:
$$\mathcal{H}^{s}_{\infty}(\tilde{\mathcal{S}_{\nu}} ) >  \omega_{s} 2^{-s} \sum_{k=1}^{\infty} (diam(E_k))^{s}  - \delta.$$
We can assume that the sets $E_k$ are open.
Since $\bigcup E_{k}$ is open, $\tilde{\mathcal{S}_{\nu}}$ is compact and $\tilde{\mathcal{S}_{\nu}} \subset \bigcup E_{k}$, we can cover $\tilde{\mathcal{S}_{\nu}}$ with finitely many $E_k$, $k=1,\ldots,K$. Letting $E$ be the union of this finite cover and $\epsilon$ be a number smaller than the minimum of the diameters of the  $E_k$'s in this finite cover, we have:
$$(\mathcal{S}_{\nu} )_{\epsilon} \subset E.$$
It follows from $\eqref{nhoodsing}$ that for $j$ large enough, we have
$$\mathcal{S}_{j} \subset E,$$
where $\mathcal{S}_{j} = \frac{\tilde{\mathcal{S}_{\mu}} - x_0}{r_j} \cap \overline{B_1(0)}$.
Hence, for $j$ large, since $\left\lbrace E_k \right\rbrace_{k=1}^{K}$ covers $\mathcal{S}_{j}$
\begin{align*}
\mathcal{H}^{s}_{\infty}(\mathcal{S}_{j} )  & \leq \omega_s 2^{-s} \sum_{k=1}^{K} (diam(E_k))^s, \\ &\leq  \mathcal{H}^{s}_{\infty}(\tilde{\mathcal{S}_{\nu}})+\delta.
\end{align*}
Since $\delta$ was chosen arbitrarily, we get $
\mathcal{H}^{s}_{\infty}(\mathcal{S}_{j} ) \leq \mathcal{H}_{\infty}^{s}(\tilde{\mathcal{S}_{\nu}})$.
Letting $j \to \infty$, we get:
$$2^{-s} \leq \limsup \mathcal{H}^{s}_{\infty}(\mathcal{S}_{j} ) \leq  \mathcal{H}_{\infty}^{s}(\tilde{\mathcal{S}_{\nu}}).$$
This gives $\mathcal{H}_{\infty}^{s}(\mathcal{S}_{\nu}) \geq  \mathcal{H}_{\infty}^{s}(\tilde{\mathcal{S}_{\nu}}) >0$. The claim is thus proved. 

Since $\mathcal{H}^{s}(\mathcal{S}_{\nu} \cap \overline{B_1(0)})>0$, by the same reasoning as for $\mu$, there exists $\tilde{\xi}$, $\tilde{\xi} \in \mathcal{S}_{\nu} \cap \overline{B_1(0)}$ such that : $$\theta^{s,*}(\mathcal{H}_{\infty}^{s} \res \mathcal{S}_{\nu}, \tilde{\xi}) \geq 2^{-s}.$$ In particular, there exists a decreasing sequence $\left\lbrace s_j \right\rbrace$ such that $\mathcal{H}^{s}_{\infty} (\mathcal{S}_{\nu} \cap \overline{B_{s_j}(\tilde{\xi}}) \geq 2^{-s} s_{j}^{s}$ and $\nu_{\tilde{\xi}, s_j} \rightharpoonup \tilde{\nu}$, where $ \tilde{\nu}$ is the normalized tangent measure to $\nu$ at $\tilde{\xi}$. 
The same procedure as above gives:
\begin{equation}\label{dimred1}
\mathcal{H}^{s}(\mathcal{S}_{ \tilde{\nu}} \cap \overline{B_1(0)})>0.
\end{equation}
Note that $\tilde{nu}$ is a conical $k$-uniform measure for some $k \leq m$ by Theorem $\ref{Preiss}$.

We repeat the procedure one final time to get a measure which is translation invariant along its one dimensional spine. Since $\mathcal{H}^{s}(\mathcal{S}_{ \tilde{\nu}} \cap \overline{B_1(0)})>0$, by the same reasoning as above we can find $\xi \in \mathcal{S}_{ \tilde{\nu}} \cap \overline{B_1(0)}$, $\xi \neq 0 $ and call $\lambda $ the normalized tangent measure to $\tilde{\nu}$ at $\xi$. We also get :
$$\mathcal{H}^{s}(\mathcal{S}_{ \lambda} \cap \overline{B_1(0)}).$$

Let $\Sigma= supp(\lambda)$. We have by Theorem $\ref{transinvariant}$ that 
$$ \Sigma = \RR \times A $$ for some $A \subset \RR^{d-1}$ such that $\mathcal{H}^{k-1} \res A$ is $(k-1)$-uniform.

So there exists $c>0$ so that $\lambda = c {\omega_{k}}^{-1} \mathcal{H}^{k} \res \left( \RR \times A \right)$ by Theorem [4.5] in $\cite{KoP}$.
By Theorem 3.11 in $\cite{KoP}$, $\lambda_{0}= \mathcal{H}^{k-1} \res A$ is $(m-1)$-uniform.

It easily follows that \begin{equation}\label{almostproduct}\mathcal{S}_{\lambda} \subset \RR \times \mathcal{S}_{\lambda_{0}}. 
\end{equation}
 from which we deduce that 
\begin{equation} \label{hausdorffdimproduct}
dim_{H}(\mathcal{S}_{\lambda}) \leq dim_{H}(S_{\lambda_0}) + 1.
\end{equation} 

But since $\mathcal{H}^{s}(\mathcal{S}_{\lambda})>0$, 
\begin{equation} \label{hausdorffdimproduct2}
dim_{H}(\mathcal{S}_{\lambda}) \geq s. 
\end{equation} 
Combining $\eqref{hausdorffdimproduct}$ and $\eqref{hausdorffdimproduct2}$, we get:
\begin{equation}
s-1 \leq dim_{H}(\mathcal{S}_{\lambda_0}).
\end{equation}
On the other hand, $S_{\lambda_0}$ being the singular set of a $(k-1)$-uniform measure, the induction hypothesis implies that $dim_{H}(\mathcal{S}_{\lambda_0}) \leq k-4 \leq m-4$. Therefore $s \leq m-3$.

\end{proof}

\end{document}